\documentclass[14]{amsart} 
 \usepackage{hyperref} 
\hypersetup{colorlinks=true}    
\usepackage{amsmath, amssymb, mathrsfs, amsfonts,  amsthm} 
\usepackage [all]{xy}

\DeclareMathOperator{\ring}{\widehat{QH}( \mathbb{CP} ^{\infty})}
\DeclareMathOperator{\qmaps}{\overline{\Omega}}

\DeclareMathOperator{\maps}{{\mathcal {P}}}

\newtheorem {theorem} {Theorem} [section]

\newtheorem {question}  [theorem] {Question}
\newtheorem {definition} [theorem] {Definition}

\newtheorem {notation}[theorem] {Notation}
\newtheorem {remark} [theorem] {Remark}
\newtheorem {rem/q}[theorem]{Remark/Question}
\numberwithin {equation} {section}

\author{Yasha Savelyev} 
\address{Centre de Recherches Math\'ematiques, Universit\'e de Montr\'eal, C.P. 6128, Succ. Centre-ville, Montr\'eal H3C 3J7, Qu\'ebec, Canada}
\email{savelyev@crm.umontreal.ca}
\title {On configuration spaces of stable maps}
\begin{document} 
 \maketitle
 \footnote{2010 Mathematics Subject Classification: 53D45, 57R19}  
 \begin{abstract} We study here some aspects of the topology of the space of smooth, stable, genus 0 curves in a Riemannian
 manifold $X$, i.e. the Kontsevich stable curves, which are not necessarily
 holomorphic. We use the Hofer-Wysocki-Zehnder polyfold structure on  this space and some natural
characteristic classes, to show that for $X=BU,$ the rational homology of the
spherical mapping space injects into the rational homology of the space of
stable curves. We also give here a definition of what we call $q$-complete
symplectic manifolds, which roughly speaking means Gromov-Witten theory captures all information about homology of the space of
smooth stable maps. 
\end{abstract}
\section {Introduction} 
The space of smooth stable curves (unparametrized stable maps) of a Riemann
surface into a Riemannian manifold appears naturally in the context Gromov-Witten theory in symplectic geometry,
particularly in the context of the beautiful polyfold approach of
Hofer-Wysocki-Zehnder \cite{Hofer.GW}.  

The topology of the configuration space of stable curves in a general
Riemannian manifold seems very interesting on its own merit. For example  we
show that the space of  based stable curves  has the structure of an $H$-space.
This is interesting as the space of unparametrized based spheres in $X$ does not have an $H$-space
structure.

Moreover, this configuration space may also be very natural in the study of
gradient flow for the energy functional on the space of smooth maps of say a
Riemann  sphere into a Riemannian manifold $X$. It was first observed by
Sacks-Uhlenbeck \cite{Uhlenbeck} that the flow lines of the resulting parabolic
flow often do not converge to smooth maps but rather the associated maps develop bubbling phenomena. This of course presents problems for Morse theory
considerations. For example, Eells and Wood \cite{minimal} show that in a
simply connected Kahler manifold $X$, the only critical points of the energy functional on the mapping space of a Riemann
sphere are (anti)-holomorphic maps, which are also absolutely energy
minimizing. If Morse theory worked as expected we could conclude that the topology of the
mapping space coincides with the topology of the space of absolute minima, i.e.
with the space  of (anti)-holomorphic maps, which is usually the wrong
conclusion. However in a series of remarkable papers \cite{SJC}, \cite{Guest1}, \cite{Guest2},
\cite{Kirwan}, \cite{Gravesen1} (sorry for incomplete list) it is shown that the
conclusion becomes essentially correct after a suitable process of stabilization. One possibility for treating or at least understanding this
problem is to partially compactify the space by adding all appropriate stable
maps, in such a way that bubbling becomes built in. One may then hope that the gradient flow on the enlarged space satisfies
some version of Palais-Smale condition, after appropriate completion. It would
be most exciting to see if the polyfold theory helps in this.

 Our main result concerns injectivity of the map from rational homology of the
spherical mapping space  into the rational homology of the space of
stable curves in the case of $X=BU (r)$. This at least tells us the space of
stable curves in that case is topologically non-trivial. Although we are far
from understanding the topology of this space. 

Finally, we define a notion of $q$-complete symplectic manifolds, inspired by
work Cohen-Segal-Jones \cite{SJC} studying the spherical mapping space into an
almost Kahler manifold by means of certain stabilization of holomorphic spherical
mapping spaces. 
\subsection {The space of smooth stable maps} \label {section.stable.maps}
Let $\maps _{0,n} ^{A} X$
denote the space whose elements are continuous maps into a 
Riemannian manifold $X,g$, with total homology class $A$, from a nodal connected
Riemann surface $\Sigma$ with genus 0, with complex structure $j$ and $n$
marked points. These maps are assumed smooth on
each component of $\Sigma$, and satisfy some additional conditions:
\begin{itemize} 
  \item If
the restriction $u _{\alpha}$ of $u$ to a component $\alpha$ is non-constant
then $u _{\alpha}$ is not null-homologous.
\item  If $u _{\alpha}$ is constant  then the number of
special points is at least 3. 
\end{itemize}
 We will distinguish one marked point by $z_0$. A pair of stable
maps $$u_1: (\Sigma_1, j_1, \{z_i ^{1}\}) \to X, u_2: (\Sigma_2,j_2, \{z_i
^{2}\}) \to X$$ are \emph{equivalent} if there is a continuous map $\phi:
\Sigma_1 \to \Sigma_2$, which is a diffeomorphism on each component, 
satisfies $\phi ^{*} j_2 = j_1$ maps marked points to marked points and
satisfies $u_2 (z_i^2) = u_1 (z_i^1)$. The quotient by the equivalence relation
will be  denoted by $\qmaps ^{A}_n X$, and its elements are called stable
curves in $X$. 

The topology on $\maps
^{A} _{0,n} X$ we consider is the standard Gromov topology, (taking into account
energy of maps for the metric $g$) and the arguments of \cite[Chapter 5]{MS} can be
adopted to show that this induces a metrizable topology on $\qmaps_n ^{A}$.
From now on we will suppress mention of the axillary
metric $g$, since topology of our spaces is independent of $ g$. This topology
should be equivalent to the topology considered in \cite[Section 3.2]{Hofer.GW}.
(At any rate the topology in \cite{Hofer.GW} is the one we  need for Theorem
\ref{thm.inj.BU}.) 

 \begin{remark} Note that under these conditions for a compact Riemannain
 manifold $X$ a bound on Riemannian energy of $u$ gives a bound on the number of
 smooth components. As by assumption that each $u _{\alpha}$ is not null-homologous, 
 it's area and consequently energy is bounded
by some constant $\hbar_g$ depending only on $g$. This is an elementary
application of the main compactness theorem for rectifiable currents in
geometric measure theory, \cite{GMT}. This brings up a question: are closed
bounded subsets in $\qmaps
^{A} _{n} X$ compact for a compact $X$? 
\end{remark}


Set $$ \qmaps ^{A} X = colim \qmaps_n ^{A} X,$$ the  colimit with
the maps in the  directed system \begin{equation} i_n:
\qmaps_n \to \qmaps _{n+1}
\end{equation}
 defined as follows. For $u \in \maps _{0,n} ^{A} X$, $u: (\Sigma, j, \{z_i\}) \to
 X$, $i_n (u)$ is an equivalence class of a map $u'$ from $$\Sigma' = \Sigma
 \sqcup (\mathbb{CP} ^{1}, 0, z'_1, z'_2)/{z_0 \sim z_1'}.$$  The map $i _{n}
 (u)$ is $u$ on the component $\Sigma$ and constant on the new component.  The new
distinguished marked point $z'_0$ on $\Sigma'$ is the marked point $0$.

 It will later be necessary to remember  more structure and
consider  the spaces $\qmaps_n X$ as  polyfolds, which is an orbifold version
of M-polyfolds,  or with even more structure as  polyfold groupoids. In this
case instead of saying that stable maps are smooth, we should instead require
some regularity on the underlying continuous maps, for details see
\cite[Section 1.1]{Hofer.GW}. 

 For the
definition and construction of these structures on spaces of stable maps of the
kind we use, the reader is referred to Hofer-Wisocki-Zehnder \cite{Hofer.GW}.
Although \cite{Hofer.GW}  constructs polyfold structures on spaces of 
stable maps in a symplectic manifold, the construction clearly works word for
word in the above setting. Although it might be at the moment mysterious why it is interesting. The present note is an attempt to explain
this. 
\subsubsection {Product operation}
 Let $$u _{1}, u _{2}: (\Sigma_1, z_0^{1}), ( \Sigma_2, z_0^2) \to X$$
be representatives for a pair of elements $|u_i|$ in $\qmaps _{n} X$,
respectively $\qmaps _{m} X$ such that $u_1 (z_0 ^{1}) = u_2 (z_0 ^{2})$. Then
we have a product $|u _{1}| \star |u _{2}| \in \qmaps _{n+m-1} X$  defined as 
an equivalence class of a map from \begin{equation*} \Sigma' = \Sigma_1 \sqcup
\Sigma_2 \sqcup (\mathbb{CP} ^{1}, 0, z'_1,z'_2)/ z_0 ^{1} \sim z_1', z_0 ^{2} \sim z_2',
\end{equation*}
which is $u_1$ respectively $u_2$ on the components $\Sigma_1$, respectively
$\Sigma_2$ and is constant on the new $ \mathbb{CP} ^{1}$ component. The new
distinguished marked point for $\Sigma'$ is  $0$. In other words we concatenate
$u_1, u_2$ with a ghost bubble as intermediary. 

This multiplication is homotopy associative, since
the main ghost components of domains for  $(u_1 \star u_2) \star u_3$, and $u_1
\star (u_2\star u_3)$, where concatenation takes place correspond to a pair of
points in $ \overline{M} _{0,4}$, which we may connect by a path. In other words
it is conceptually the same argument as the argument for associativity of
quantum multiplications. Also the maps 
\begin{align*} \qmaps _{n} X \times \qmaps_m X \to \qmaps _{n+m-1} X
\xrightarrow{i _{n+m} \circ i _{n+m-1}} \qmaps _{n+m+1} X\\
\qmaps _{n} X \times \qmaps _{m} X \xrightarrow{i _{n} \times i _{m}} \qmaps
_{n+1} X \times \qmaps  _{m+1} X \to \qmaps _{n+m+1} X, 
\end{align*}
are homotopy equivalent by a similar argument. Consequently there is
 an induced  map in the homotopy
category $\qmaps X \times \qmaps X \to \qmaps X$. 

 If we ask that our maps $u$ are based, i.e. map the
distinguished marked point $z_0$ to $x_0 \in X$, then the corresponding space
$\qmaps _{x_0} X $ is a homotopy associative H-space. Consequently homology
of $\qmaps _{x_0} X$ is a ring with Pontryagin product.
 \begin{notation} From
now on we will be in the above based situation and the subscript $x_0$ in $\qmaps _{x_0}$ will be dropped. 
\end{notation}
Our main observation in this paper is this:
\begin{theorem} \label{thm.inj.BU} The natural map $H_*(\Omega ^{2} BSU (r),
\mathbb{Q}) \to Cob_* ^{orb}(\qmaps BSU (r), \mathbb{Q})$, is injective for $*
\leq 2r-2$.
\end{theorem} 
On the right hand side we have orbifold bordism
groups which are to be defined.
\subsection {Complete symplectic manifolds}
Suppose now $(X, \omega)$ is a symplectic manifold.
%
Let $a_i: D _{i} \to X$, $a _{\xi}: D _{\xi} \to \overline M _{0,n}$ be
smooth maps of closed oriented smooth manifolds, with $\overline M _{0,n}$
denoting the moduli space of stable genus 0 Riemann surfaces with $n$ marked
points.

 Under suitable conditions, for example if
$(X, \omega)$ is semi-positive we have natural cycles $gw: \overline{\mathcal {M}} _{n} (A, \{a_i\}, a_{\xi}) \to
 \qmaps X$, defined as follows. Consider the diagram below:
 \begin{equation} \xymatrix{  & (\prod _{i}
D_i) \times D _{\xi} \ar[d] ^{prod} \\ \overline{\mathcal {M}} _{0,n} ^{A} (X,J)
\ar[r] & X ^{n} \times \overline{M} _{0,n}, }
\end{equation}
with $ \overline {\mathcal {M}} _{0,n} ^{A} (X,J)$ denoting the compactified
moduli space of genus zero, class $A$, $J$-holomorphic curves in $X$, for a
regular $\omega$-tamed $J$. After perturbing the maps to be transverse, we
define $\overline{\mathcal {M}} _{n} ^{A}(X, \{a_i\}, \xi)$ as the pull-back of this diagram (oriented fibre product) and the cycle $gw$
is defined to be the composition of the projection of $\overline{\mathcal {M}}
_{n} ^{A}(X, \{a_i\}, \xi)$ to $ \overline {\mathcal {M}} _{0,n} ^{A} (X;J)$, with the tautological map
\begin{equation*}  \overline {\mathcal {M}} _{0,n} ^{A} (X,J) \to \qmaps X.
\end{equation*} 
 For a completely general symplectic manifold $(M,
\omega)$ the homology class of the cycle $gw$ is defined via the homology
pushforward of the orbifold virtual fundamental class of $\overline{\mathcal {M}} _{n} (A, \{a_i\}, \xi)$. We shall call  cycles
$gw$: \emph { \textbf{Gromov-Witten}} cycles. We will also call all 
0-dimensional cycles into $\qmaps X$ Gromov-Witten cycles. 

One of the motivations we
had for undertaking study of configuration space of smooth stable maps, is  so
we could make the following definition, we say more in the remark below. \begin{definition} We
will say that a symplectic manifold $ (X, \omega)$ is \emph { \textbf{q-complete}} if homology
of $\qmaps _{} X$ is multiplicatively generated over $ \mathbb{Q}$ by homology
classes of Gromov-Witten cycles.
\end{definition}
\begin{remark} This is  the
homological version of homotopy approximation of $\Omega ^{2} X$ by 
holomorphic mapping spaces  from $ \mathbb{CP} ^{1}$, which was studied for
example in \cite{SJC}, \cite{Guest1}, \cite{Guest2}, \cite{Kirwan},
\cite{Gravesen1}. Given a symplectic manifold $X,\omega$, the
basic question is when does \begin{equation} \label {eq.limit} \Omega ^{2} X
\simeq Hol^+  ( \mathbb{CP} ^{1}, X, \omega, j_X)),
\end{equation}
where $Hol ^+$ denotes the group completion of the topological monoid of based
$j _{X}$-holomorphic maps of $ \mathbb{CP} ^{1}$ into $X$, under the gluing
operation.
 Remarkably, this is known to be the case for
example, for complex projective spaces, generalized flag manifolds and toric manifolds. Unfortunately \eqref{eq.limit} 
only makes sense for a fixed complex structure $j_X$ i.e. it is
a priori not a symplectic property. We wanted a purely symplectic notion, and
$q$-complete symplectic manifolds is one possibility.
%
\end{remark}
\begin{question} Is $ \mathbb{CP} ^{n}, \omega_{st}$  q-complete?
\end{question}
This appears to be at the moment a difficult and interesting question. Some
interesting and possibly related work is done by Miller \cite{miller}. 
 \subsection* {Acknowledgements} I am grateful to Weimin
Chen and Dusa McDuff for discussions.
\section {Main argument} \label{sect.main}
\subsection {Quantum classes}
We are going to use the basic notation and definitions of \cite{BP}. 
We will be concerned with stable quantum classes, which
are classes
\begin{equation}  \label {eq.qclasses}
qc ^{\infty} _{k} \in H ^{2k} (\Omega ^{2} BSU (r) \simeq \Omega SU (r),
QH ( \mathbb{CP} ^{r-1})),
\end{equation}
with $2k \leq 2r -2$. The quantum homology ring $QH ( \mathbb{CP} ^{r-1})$ is
taken with $ \mathbb{Q}$ coefficients, so that additively it is just the
rational homology of $ \mathbb{CP} ^{r-1}$. 
 The above classes are
interesting in the present context because they canonically ``extend'' to
$\qmaps  BSU (r)$, while such an extension is not apparent for Chern classes
especially since they are not \emph{intrinsically} defined on $\Omega ^{2} BSU
(r)$. Extend here means that they are pull-backs of some orbifold cohomology
classes on $\qmaps BSU (r)$, via  the natural map $\Omega ^{2} BSU (r)\to \qmaps
BSU (r)$. Although orbifold cohomology groups for us will just be the space of
$\ring$ valued linear functionals on certain orbifold bordism groups $Cob_* 
^{orb}(\qmaps BSU (r))$,  which contain too much  information to be really
practical, and in principle it would be good to cut this down, but it is not
obvious how to do this in a way that still allows definition of our classes. In particular it is at the moment unclear  how and
 if these are related to Chen-Ruan orbifold cohomology groups (with suitable
 coefficients).
A map $X \to \qmaps BSU (r)$ by definition factors through a map $X \to \qmaps
_{n} BSU (r)$ for some $n$, composed with the universal map  $$\qmaps
_{n} BSU (r)  \to \qmaps BSU = colim _{n} \qmaps _{n} BSU (r).$$
Consequently, we may define a map from a smooth manifold $f: X \to \qmaps BSU
(r)$ as sc-smooth if factors through an sc-smooth map to $\qmaps _{n} BSU (r)$.
Note that $BU (r) = \lim _{n \mapsto \infty} Gr _{ \mathbb{C}} (r, \mathbb{C}
^{n}) $, and we of course only need a polyfold structure on $ \qmaps _{n} Gr _{ \mathbb{C}} (r, \mathbb{C}
^{n}) $ for every $n$, for all our arguments. To simplify notation we will
still just refer $ \qmaps BSU (r)$ as a polyfold, although this of course only
means that it is a direct limit of polyfolds. 

 The group $Cob _{k} ^{orb} (\qmaps BSU (r))$, is the group of
equivalence classes of sc-smooth orbifold maps  $f: X ^{k} \to \qmaps BSU(r)$,
with $X ^{k}$ closed oriented smooth orbifold with corners, with
composition given by disjoint union, where the equivalence relation is $(f _{1}, X_1) \sim (f _{2}, X _{2})$ if there
is an sc-smooth map of an orbifold with boundary:
\begin{equation*} F: B ^{k+1} \to \qmaps BSU (r),
\end{equation*}
with $(\partial F, \partial B) = (f_1, X_1 ^{op}) \sqcup (f_2, X _{2})$, where
$X _{1} ^{op}$ denotes $X _{1}$ with the opposite orientation. 
 From the groupoid point of view of orbifolds  an
orbifold $X$ is some coarse equivalence class of an \`{e}tale, proper, stable,
smooth groupoid, also known as orbifold groupoid, $\mathcal {X}$, see for
example \cite{dusa.orb}. 
The orbifold map $f$ above is determined by data of an sc-smooth functor $
\widetilde{f}: \mathcal {X} \to \qmaps BSU (r)$, with the right side considered
as a polyfold groupoid, also called ep-groupoid, but we only use the former name. Let
$E \to BSU (r)$ denote the projectivization of the universal $ \mathbb{C}
^{r}$-bundle. 

We have a natural polyfold groupoid fibration $\mathcal {E} \to\qmaps BSU (r)$, 
with fiber over $x \in \qmaps BSU (r)$, $[x: \Sigma _{x} \to BSU (r)]$ the space
of  ``stable sections" i.e. section class stable maps $\sigma: \Sigma'_x \to E_x=x ^{*} E$.
These are defined as follows:   $\Sigma' _{x}$ has some components labeled as
principal, and some components labeled as vertical. The principal components of
$\Sigma'_x$ are identified with components of $\Sigma _{x}$ and $ \sigma$ is a
``smooth'' (continuous section with appropriate regularity
\cite[Section 1.1]{Hofer.GW}) section of $E_x$ over these components, while the
vertical components of $\Sigma'_x$ are mapped into the fibers of $ E_x$. 
The restriction of $\sigma$ to the collection of all vertical components is required to be a stable map as defined in Section \ref{section.stable.maps}. 

 This definition requires a bit more explanation as
$x$ is only some equivalence class of a map. To this end consider the large category
 $ \mathcal {P} ^{A} _{0,n} BSU (r)$ with morphisms  given by label preserving
 reparametrizations, (see Section \ref{section.stable.maps}).  Then clearly we have an analogously defined
fibration of topological categories $ \mathcal {E} \to \mathcal {P} ^{A} _{0,n}
BSU (r)$  (and now there is no ambiguity in the definition). On the other hand
the groupoid $ \qmaps BSU (r)$ is actually constructed as a refinement of $\mathcal {P} ^{A} _{0,n} BSU (r)$ and this
refinement also refines the above categorical fibration to a polyfold groupoid
fibration  $pr: \mathcal {E} \to\qmaps BSU (r)$. In what follows we keep
applying this implicit understanding. 

Fix a unitary (in other words $ PU (r)$) connection $ \mathcal {A}$ on $E \to
BSU (r)$. Then the restriction $ \mathcal {A} _{x}$ of $ \mathcal {A}$ to $E_x = x ^{*} E$
 induces  almost complex structures $  \{J_x\}$   in the following
 standard way.
\begin{itemize} 
  \item The natural map 
$\pi: (E_x, J _{ x})
\to (\Sigma _{x}, j_x)$ is holomorphic.  
\item  $J _{ x}$ preserves the horizontal subbundle of $TE _{x}$
induced by $ \mathcal {A}_x$.
\item $J _{x}$ preserves the vertical tangent bundle $T ^{vert} E _{x}$ of $
\mathbb{CP} ^{r-1} \hookrightarrow E_x \to \Sigma _{x}$, and
restricts to the standard complex structure on the fibers $ \mathbb{CP} ^{r-1}$.
(That is to say the fibers are identified with $ \mathbb{CP} ^{r-1}$ up to
action of $PU (r)$, which preserves this complex structure.)
\end{itemize}

Let $ \widetilde{ \mathcal {E}}$ denote the pull-back of $ \mathcal {E}$ to $
\mathcal {X}$, via $ \widetilde{f}: \mathcal {X} \to \qmaps BSU (r)$.  Then
over $ \widetilde{ \mathcal {E}}$ we have a natural strong Polyfold Banach
bundle $ \mathcal {W}$. For an element $\sigma \in E _{x}$, $\sigma: \Sigma'
_{x} \to E _{x}$, the fiber over
$\sigma$ consists (after appropriate completion) of the space of continuous,
and smooth over smooth components $J _{x}$ anti complex linear 1-forms on
$\Sigma'_x$ with values in $\sigma ^{*} T ^{vert} E _{x}$.
By essentially identical arguments to \cite[Section 1.2]{Hofer.GW} over the
whole $ \mathcal {E}$ this can be given the structure of a strong polyfold Banach bundle. 
 And  we have a Fredholm sc-section of $ \mathcal {W}$: the Cauchy-Riemann
 section, by taking the $J_x$ anti-linear part of the differential of $\sigma=
 (\Sigma'_x \to E _{x})$ as a map into $\sigma ^{*} T ^{vert} E _{x}$.  
 We finally define our orbifold cohomology
classes  $$qc_k \in H ^{k} (\qmaps BSU (r), QH ( \mathbb{CP} ^{r-1})),$$ 
which to remind the reader for us are just $QH ( \mathbb{CP} ^{r-1})$ valued
linear functionals on $Cob _{k} ^{orb} (\qmaps BSU (r))$. We will restrict our
discussion to the identity component of $ \qmaps BSU (r)$. We need
to say how to compute 
\begin{equation} \label {eq.correlator} \langle qc _{k}, [ \widetilde{f}]
\rangle,
\end{equation}
where $ \widetilde{f}$ denotes the functor $ \widetilde{f}: \mathcal {X} \to
\qmaps BSU (r)$ lifting the data of an sc-smooth orbifold map $f: X \to \qmaps
BSU (r)$.

Recall that each $\Sigma _{x}$ comes with a distinguished marked point $z_0$,
which is mapped to a fixed  base point in $BSU (r)$, (recall definition of
$\qmaps BSU (r)$). Consequently, we have a smooth family of embeddings $I_x: 
\mathbb{CP} ^{r-1} \to E _{x}$, which takes $\mathbb{CP} ^{r-1}$ to the
fiber of $E _{x}$ over $z_0$.  Let
$ \overline{ \mathcal {M}} ( \mathcal {P}, [ \mathbb{CP} ^{l}], d)$ denote the
polyfold groupoid consisting  of pairs $
(\sigma,x)$, $x \in \mathcal {X}$ with $\sigma \in \widetilde{ \mathcal {E}}$ in
the 0-set of the above constructed Cauchy-Riemann section, whose total
degree is $d$, and which intersects $I_x ( \mathbb{CP} ^{l})$. Where $d$
 here is defined as 
 \begin{equation*} \langle c _{1} (T ^{vert} E _{x}), \sigma \rangle \cdot
 \frac{1}{r}.
\end{equation*}
 This is  an integer, since by assumption  $ \widetilde{f}$ maps to
 the identity component of $ \qmaps BU (r)$ which implies that  $P _{x}  \simeq
 \mathbb{CP} ^{r-1} \times \Sigma _{x}$, as a topological bundle.  But then the
 vertical Chern number of a section is $n$ times the degree of the
 projection of the section to $ \mathbb{CP} ^{r-1}$.

The virtual dimension of $\overline{ \mathcal
{M}} ( \mathcal {P}, [ \mathbb{CP} ^{l}], d)$, is $2 d r + 2k + 2l $.
Note that if $0<2k \leq 2r-2$, then
 \begin{equation*}  2 dr + 2k + 2l <0 \text { unless $d \geq -1$}.
\end{equation*} 
On the other hand $d> 0$ results in too high virtual dimension, and $d=0$ only
contributes to degree 0 class as we will see below. 
Under this condition definition of quantum classes is particularly simple:
 \begin{equation*} \langle qc
_{k}, [ \widetilde{f}] \rangle = b  \in QH ( \mathbb{CP} ^{r-1}),
\end{equation*}
where $b \in H_* ( \mathbb{CP} ^{r-1}, \mathbb{Q}) $ is defined by duality:
\begin{equation*} b \cdot [\mathbb{CP} ^{l}] = \# \overline{ \mathcal {M}} (
\mathcal {P}, [ \mathbb{CP} ^{l}], -1) \in\mathbb{Q},
\end{equation*}
where the left side is the usual intersection product and the right side is the
orbifold Gromov-Witten invariant counting elements of $  \overline{ \mathcal {M}} ( \mathcal {P}, [
\mathbb{CP} ^{l}], -r)$, which is zero unless the expected dimension is zero. 
Of course we also have to show
that our definition of $qc_k$ is well defined, i.e. that \eqref{eq.correlator}
is independent of the choice of representative $ \widetilde{f}$. This is worked
out in \cite{GS}, and readily generalizes to
polyfold setting, as it is just the usual cobordism of moduli space argument,
similarly with independence of choices of almost complex structures,
(connections).

 The pull-back of the above defined classes to $\Omega ^{2} BSU (n)$ are
 exactly the stable quantum classes considered in \cite{BP}, except
 of course we did not need orbifold cycles but worked with cycles that are maps of closed
 oriented manifolds. 
  We paraphraze the main theorem of \cite{BP} as follows:
  \begin{theorem} [\cite {BP}] \label{thm.bp} If $2k \leq 2r-2$, then $0 \neq a
  \in H _{2k} (\Omega ^{2} BSU (r), \mathbb{Q})$ if and only if for some $
  \{\beta_i, \alpha_i\}$ \begin{equation*} 0 \neq \langle \prod _{i} qc
  _{\beta_{i}} ^{\alpha_i}, a \rangle, \text { where} \sum _{i} {2 \beta _{i}} \cdot {\alpha _{i}} = k.
\end{equation*}  
  \end {theorem}
\begin{proof}[Proof of Theorem \ref{thm.inj.BU}] By Milnor-Moore,
Cartan-Serre theorems the rational homology of $\Omega ^{2} BSU (r) \simeq
\Omega SU (r)$ is multiplicatively generated  by spherical  
classes, in particular by maps of smooth manifolds.  Consequently the theorem
follows immediately by Theorem \ref{thm.bp} and by existence of extension of
classes $qc _{k}$ to $ \qmaps BSU (r)$ described above.
\end{proof}
\bibliographystyle{siam}  
\bibliography{link} 

\begin{thebibliography}{10}

\bibitem{SJC}
{\sc R.~L. Cohen, J.~D.~S. Jones, and G.~B. Segal}, {\em Stability for
  holomorphic spheres and {M}orse theory}, Contemporary Mathematics,  (2000),
  pp.~87 -- 106.

\bibitem{minimal}
{\sc J.~Eells and J.~Wood}, {\em {Maps of minimum energy.}}, J. Lond. Math.
  Soc., II. Ser., 23 (1981), pp.~303--310.

\bibitem{GMT}
{\sc H.~Federer}, {\em Geometric measure theory}, no.~153 in Die Grund. der
  math. Wissenschaften, Springer, New York, 1969.

\bibitem{Gravesen1}
{\sc J.~Gravesen}, {\em {On the topology of spaces of holomorphic maps.}}, Acta
  Math., 162 (1989), pp.~247--286.

\bibitem{Guest2}
{\sc M.~Guest}, {\em {Topology of the space of absolute minima of the energy
  functional.}}, Am. J. Math., 106 (1984), pp.~21--42.

\bibitem{Guest1}
{\sc M.~A. Guest}, {\em {The topology of the space of rational curves on a
  toric variety.}}, Acta Math.,  (1995).

\bibitem{Hofer.GW}
{\sc H.~Hofer, W.~Kris, and E.~Zehnder}, {\em {Applications of Polyfold Theory
  I: The Polyfolds of Gromov-Witten Theory}}, arXiv:1107.2097,  (2011).

\bibitem{Kirwan}
{\sc F.~Kirwan}, {\em {On spaces of maps from Riemann surfaces to Grassmannians
  and applications to the cohomology of moduli of vector bundles.}}, Ark. Mat.,
  24 (1986), pp.~221--275.

\bibitem{dusa.orb}
{\sc D.~McDuff}, {\em {Groupoids, branched manifolds and multisections.}}, J.
  Symplectic Geom., 4 (2006), pp.~259--315.

\bibitem{MS}
{\sc D.~McDuff and D.~Salamon}, {\em $J$--holomorphic curves and symplectic
  topology}, no.~52 in American Math. Society Colloquium Publ., Amer. Math.
  Soc., 2004.

\bibitem{miller}
{\sc J.~Miller}, {\em {Homological stability properties of spaces of rational
  J-holomorphic curves in $P^2$}}, arXiv:1110.1665,  (2011).

\bibitem{Uhlenbeck}
{\sc J.~Sacks and K.~Uhlenbeck}, {\em {The existence of minimal immersions of
  2-spheres.}},  (1981).

\bibitem{GS}
{\sc Y.~Savelyev}, {\em Quantum characteristic classes and the {H}ofer metric},
  Geometry Topology,  (2008).

\bibitem{BP}
\leavevmode\vrule height 2pt depth -1.6pt width 23pt, {\em Bott periodicity and
  stable quantum classes}, arXiv:0912.2948,  (2012).

\end{thebibliography}
\end{document}